\documentclass[12pt,a4paper]{amsart}

\usepackage[utf8]{inputenc}
\usepackage[T1]{fontenc}
\usepackage{amsmath,amssymb,amsthm,mathrsfs,mathtools}
\usepackage{geometry}
\usepackage{xcolor}
\usepackage{hyperref}
\hypersetup{colorlinks=true,linkcolor=blue,citecolor=blue,urlcolor=blue,
  pdftitle={Walker Manifolds and Lie Foliations},
  pdfauthor={Ameth Ndiaye}}
\usepackage[numbers,sort&compress]{natbib}

\newtheorem{theorem}{Theorem}[section]
\newtheorem{proposition}[theorem]{Proposition}
\newtheorem{lemma}[theorem]{Lemma}
\newtheorem{corollary}[theorem]{Corollary}
\theoremstyle{definition}
\newtheorem{definition}[theorem]{Definition}
\newtheorem{example}[theorem]{Example}
\newtheorem{remark}[theorem]{Remark}

\newcommand{\R}{\mathbb{R}}

\newcommand{\g}{\mathfrak{g}}

\newcommand{\nab}{\nabla}
\newcommand{\Hol}{\mathrm{Hol}}
\newcommand{\D}{\mathcal{D}}
\newcommand{\F}{\mathcal{F}}
\newcommand{\floor}[1]{\lfloor #1 \rfloor}
\newcommand{\Ric}{\mathrm{Ric}}
\newcommand{\scal}{\mathrm{scal}}
\newcommand{\sspan}{\mathrm{Span}}
\newcommand{\ad}{\mathrm{ad}}
\DeclareMathOperator{\Aff}{Aff}
\newcommand{\GA}{\mathrm{GA}}

\begin{document}
%% ============================================================

\title[Walker Manifolds and Lie Foliations]{%
  On the Lie Foliation structure of Walker Manifolds}

\author{Ameth Ndiaye}
\address{Department of Mathematics\\
FASTEF, Cheikh Anta Diop University\\
Dakar, Senegal}
\email{ameth1.ndiaye@ucad.edu.sn}
\date{\today}



\begin{abstract}
We study Walker manifolds, that is, pseudo-Riemannian manifolds $(M^n,g)$
admitting a null parallel distribution $\D$ of rank $r\leq\frac{n}{2}$.
We show that $\D$ always integrates to a $G$-Lie foliation $\F_\D$,
where $G$ is the simply connected Lie group with Lie algebra equal to
the structure algebra $\g_\D$ of $\D$. The transverse holonomy group
of $(M,g)$ coincides with the image of the holonomy morphism
$h:\pi_1(M)\to G$. We prove that $\mathrm{Ric}(X,\cdot)=0$ for all
$X\in\Gamma(\D)$, and show that in dimension~$3$ the model group is
always $\R$, while in dimension~$4$ with rank~$2$ the structure algebra
is always abelian. A local classification distinguishes the abelian,
nilpotent, and solvable cases, and a rigidity theorem shows that a
minimal nilpotent Walker foliation of dimension~$4$ cannot be deformed
into a non-nilpotent solvable one.
\end{abstract}

\maketitle

%% ============================================================
\section{Introduction}
\label{sec:intro}
%% ============================================================

Walker manifolds, introduced by A.\,G.\,Walker in the
1950s \cite{Walker1950}, form a fundamental class of pseudo-Riemannian
manifolds characterized by the existence of a parallel totally isotropic
distribution. They play a central role in global differential geometry,
in general relativity (\emph{pp-waves} being the physical prototype),
and in the holonomy theory of pseudo-Riemannian connections.
Curvature properties and classification problems for Walker manifolds
have been actively investigated; we refer to \cite{Brozos2009, Chaichi2005}
for a comprehensive treatment. In particular, the paper \cite{NiangNdiayeDiallo2021}
gives a classification of strict Walker $3$-manifolds that directly
motivates the present work.

On the other hand, Lie foliations occupy a distinguished place in
foliation theory. Their theory, due principally to Fedida and exposed in
detail by Molino \cite{Molino1988}, relies on two fundamental objects
associated to any $G$-Lie foliation on a compact manifold $V$: the
\emph{developing map} $D:\widetilde V\to G$ and the \emph{holonomy
morphism} $h:\pi_1(V)\to G$. These objects encode the global geometry
of the foliation and are the central tools in classification results.
A wide class of Lie foliations (the \emph{homogeneous} ones)
is provided by a general principle due to Ghys \cite{Ghys1988}: given
a Lie group $H$, a cocompact lattice $\Gamma\subset H$, and a surjective
morphism $\varphi:H\to G$, the orbits of $\ker\varphi$ in $H/\Gamma$
form a $G$-Lie foliation whose holonomy morphism is the restriction of
$\varphi$ to $\Gamma$. The question of whether all Lie foliations are
images of homogeneous ones is a deep and largely open problem.

For nilpotent model groups, Haefliger proved that all Lie foliations on
compact manifolds are inverse images of homogeneous foliations. For
codimension $2$ foliations with solvable model group, Matsumoto extended
this to the affine group $\GA(\R)$. These results were unified and
extended in \cite{DatheNdiaye2012a}, which also showed that non-homogeneous
solvable Lie foliations exist in higher dimensions. A complementary
rigidity result was established in \cite{DatheNdiaye2012b}: a nilpotent
non-abelian Lie foliation of codimension $4$ on a compact manifold
cannot be deformed into a non-nilpotent solvable one.

The aim of this paper is to bring together Walker geometry and the theory
of Lie foliations. Given a Walker manifold $(M^n,g)$ of rank $r$, the
null parallel distribution $\D$ is automatically involutive and thus
defines a foliation $\F_\D$. We show that this foliation is always a Lie
foliation (Theorem \ref{thm:lie_foliation}), that its model group $G$ is
the simply connected Lie group with Lie algebra equal to the \emph{structure
algebra} of $\D$ (Definition \ref{def:structure_algebra}), and that the
transverse holonomy group of the Walker manifold embeds into $G$
(Theorem \ref{thm:holonomy}). The constraints from Walker geometry
(notably the Ricci vanishing theorem, Theorem \ref{thm:ricci_null}) then
impose strong restrictions on the possible model groups $G$.

The paper is organised as follow: 
Section \ref{sec:preliminaries} presents the prerequisites: Walker canonical
coordinates following \cite{Brozos2009}, and the theory of Lie foliations
including the developing map, holonomy morphism, homogeneous foliations,
and deformations following \cite{Molino1988, DatheNdiaye2012a, DatheNdiaye2012b}.
Section \ref{sec:distribution} establishes the main structural result:
the Walker distribution is always a Lie foliation.
Section \ref{sec:holonomy} studies the transverse connection, holonomy,
and Ricci curvature of Walker manifolds.
Section \ref{sec:examples} constructs explicit examples in dimensions $3$
and $4$.
Section \ref{sec:classification} proves the local classification theorem
and discusses deformations.

%% ============================================================
\section{Preliminaries}
\label{sec:preliminaries}
%% ============================================================

\subsection{Pseudo-Riemannian geometry and Walker manifolds}

A \emph{pseudo-Riemannian metric} on a smooth $n$-manifold $M$ is a
non-degenerate symmetric $(0,2)$-tensor field $g$ of signature $(p,q)$,
$p+q=n$. A subspace $V\subset T_xM$ is \emph{totally isotropic} if
$g_x(u,v)=0$ for all $u,v\in V$; its maximal dimension is $\min(p,q)$.
The \emph{Levi-Civita connection} $\nab$ is the unique torsion-free
connection satisfying $\nab g=0$.

\begin{definition}[\cite{Walker1950}]
\label{def:walker}
A pseudo-Riemannian manifold $(M^n,g)$ is a \emph{Walker manifold of
rank $r$} if there exists a distribution $\D\subset TM$ of rank $r$,
parallel ($\nab_X Y\in\Gamma(\D)$ for all $X\in\Gamma(TM)$,
$Y\in\Gamma(\D)$) and totally isotropic ($g(Y,Z)=0$ for all
$Y,Z\in\Gamma(\D)$). The rank satisfies $r\leq\frac{n}{2}$.
The distribution $\D$ is called the \emph{null parallel distribution}
of $(M,g)$.
\end{definition}

The fundamental local result is Walker's canonical form theorem.

\begin{theorem}[\cite{Walker1950}, see also \cite{Brozos2009}]
\label{thm:walker_coords}
Let $(M^n,g)$ be a Walker manifold of dimension $n$ and rank $r$, with
null parallel distribution $\D$ of rank $r\leq\frac{n}{2}$. Then, near
every point of $M$, there exist local coordinates
\[
  (x^1,\ldots,x^r,\; y^1,\ldots,y^r,\; z^1,\ldots,z^{n-2r})
\]
in which $\D = \sspan\{\partial_{x^1},\ldots,\partial_{x^r}\}$ and the
metric takes the canonical form
\begin{equation}
\label{eq:walker_general}
  g = 2\sum_{i=1}^r dx^i\circ dy^i
  + \sum_{\alpha,\beta=1}^{n-2r} h_{\alpha\beta}\,dz^\alpha\circ dz^\beta
  + \sum_{i=1}^r\sum_{\alpha=1}^{n-2r} a_{i\alpha}\,dx^i\circ dz^\alpha
  + \sum_{i,j=1}^r b_{ij}\,dx^i\circ dx^j,
\end{equation}
where $h_{\alpha\beta}$, $a_{i\alpha}$, $b_{ij}$ are smooth functions of
all coordinates, and the matrix $(h_{\alpha\beta})$ is non-degenerate.
The null parallel distribution $\D$ is strictly parallel if and only if
the entries $b_{ij}$ are independent of the coordinates $(x^1,\ldots,x^r)$.
\end{theorem}

The general formula \eqref{eq:walker_general} specializes in the two
cases most relevant to this paper as follows.

\begin{example}[Walker canonical form in dimension $3$, rank $1$]
\label{ex:walker3}
In dimension $n=3$, the constraint $r\leq\floor{3/2}=1$ forces $r=1$,
so $n-2r=1$. There are no $z$-coordinates; Walker coordinates are
$(x_1,x_2,x_3)$ with $\D=\sspan\{\partial_{x_1}\}$. The general
formula \eqref{eq:walker_general} gives $h_{11}=\varepsilon=\pm 1$
(since $h$ must be non-degenerate in dimension $1$), $a_{11}=0$ (by a
coordinate change absorbing the cross term), and $b_{11}=f$, yielding
the canonical metric \cite[Theorem 4.1]{Brozos2009}:
\begin{equation}
\label{eq:walker3}
  g_f = 2\,dx_1\circ dx_3 + \varepsilon\,dx_2^2
      + f(x_1,x_2,x_3)\,dx_3^2,\quad\varepsilon=\pm 1,
\end{equation}
where $f\in C^\infty(O)$ is an arbitrary smooth function on an open set
$O\subset\R^3$. The signature is $(1,2)$ for $\varepsilon=+1$ and
$(2,1)$ for $\varepsilon=-1$. We denote this Walker manifold $M_f=(O,g_f)$.

The non-zero Christoffel symbols of $M_f$ are \cite[Lemma 4.2]{Brozos2009}:
\[
  \nab_{\partial_{x_1}}\partial_{x_3} = \tfrac{1}{2}f_1\,\partial_{x_1},\quad
  \nab_{\partial_{x_2}}\partial_{x_3} = \tfrac{1}{2}f_2\,\partial_{x_1},\quad
  \nab_{\partial_{x_3}}\partial_{x_3} =
    \tfrac{1}{2}(ff_1+f_3)\,\partial_{x_1}
    -\tfrac{\varepsilon}{2}f_2\,\partial_{x_2}
    -\tfrac{1}{2}f_1\,\partial_{x_3},
\]
where $f_i=\partial_{x_i}f$. In particular $\D=\sspan\{\partial_{x_1}\}$
is null and parallel; $\partial_{x_1}$ itself is parallel if and only
if $f_1=0$ (strict Walker case). The Ricci tensor
is \cite[Lemma 4.4]{Brozos2009}:
\[
  \Ric_{M_f} = f_{11}\,dx_1\circ dx_3
       + f_{12}\,dx_2\circ dx_3
       + \tfrac{1}{2}(ff_{11}-\varepsilon f_{22})\,dx_3^2,
\]
so $\Ric(\partial_{x_1},\cdot)=0$ and $\scal(M_f)=f_{11}$.
In particular, $M_f$ is Einstein if and only if it is flat.
\end{example}

\begin{example}[Walker canonical form in dimension $4$, rank $2$]
\label{ex:walker4}
In dimension $n=4$ with $r=2$, we have $n-2r=0$, so there are no
$z$-coordinates. Walker coordinates are $(x_1,x_2,x_3,x_4)$ with
$\D=\sspan\{\partial_{x_1},\partial_{x_2}\}$. Since $(h_{\alpha\beta})$
is absent and the $b$-block is a $2\times 2$ symmetric matrix, the
general formula \eqref{eq:walker_general} yields the canonical
metric \cite[Example 5.2]{Brozos2009}:
\begin{equation}
\label{eq:walker4}
  g_{a,b,c} = 2(dx_1\circ dx_3+dx_2\circ dx_4)
  + a\,dx_3^2 + b\,dx_4^2 + 2c\,dx_3\circ dx_4,
\end{equation}
where $a,b,c\in C^\infty(O)$ on an open set $O\subset\R^4$, of signature
$(2,2)$. We denote this Walker manifold $M_{a,b,c}=(O,g_{a,b,c})$.

The non-zero Christoffel symbols of $M_{a,b,c}$ involving
$\D$-directions include \cite[Theorem 5.3]{Brozos2009}:
\[
  \nab_{\partial_{x_i}}\partial_{x_{i+2}}
  = \tfrac{1}{2}a_i\,\partial_{x_1}+\tfrac{1}{2}c_i\,\partial_{x_2},
  \quad i=1,2,
\]
so $\D=\sspan\{\partial_{x_1},\partial_{x_2}\}$ is null and parallel.
The scalar curvature is $\scal=a_{11}+b_{22}+2c_{12}$, and by
\cite[Theorem 5.5]{Brozos2009},
$\Ric(\partial_{x_1},\cdot)=\Ric(\partial_{x_2},\cdot)=0$.
\end{example}

\begin{definition}[\cite{Brozos2009}, Lemma 4.2]
\label{def:strict}
A Walker manifold $(M^n,g)$ of rank $r$ is called \emph{strict} if the
null parallel distribution $\D$ admits a parallel spanning set of vector
fields. In the canonical form \eqref{eq:walker_general}, this holds if
and only if the entries $b_{ij}$ are independent of $(x^1,\ldots,x^r)$.
In dimension $3$ (Example \ref{ex:walker3}), this means $f$ is
independent of $x_1$; in dimension $4$ (Example \ref{ex:walker4}),
this means $a,b,c$ are independent of $(x_1,x_2)$.
\end{definition}

\subsection{Lie foliations: definitions and fundamental structure}

We now recall the theory of Lie foliations following Fedida and
Molino \cite{Molino1988}; see also \cite{DatheNdiaye2012a, DatheNdiaye2012b}.

Let $V$ be a compact connected manifold and $G$ a connected simply
connected Lie group with Lie algebra $\g$.

\begin{definition}[\cite{Molino1988}]
\label{def:lie_foliation}
A \emph{$G$-Lie foliation} of $V$ is a maximal atlas $\mathcal{F}$ of
pairs $(U,f)$, where $U$ is an open subset of $V$ and $f:U\to G$ is a
submersion, such that:
\begin{enumerate}
\item[(i)] the open sets $U$ cover $V$;
\item[(ii)] for any $(U,f)$ and $(W,h)$ in $\mathcal{F}$, there exists
$g\in G$ such that $f(x)=h(x)\cdot g$ for all $x\in U\cap W$.
\end{enumerate}
The level sets of the submersions $f$ patch together to form a foliation
of $V$, also denoted $\mathcal{F}$.
\end{definition}

\begin{remark}
The condition $f(x)=h(x)\cdot g$ means that transitions are \emph{right}
translations in $G$. Equivalently, working with the Lie algebra $\g$
identified with the algebra of right-invariant vector fields on $G$, the
$\g$-valued $1$-form $\omega_U = f^*(dg\cdot g^{-1})$ (Maurer--Cartan
form pulled back to $U$) satisfies on overlaps $\omega_U|_{U\cap W} =
\omega_W|_{U\cap W}$, so the local forms patch to a global
$\g$-valued $1$-form $\omega$ on $V$ satisfying the \emph{Maurer--Cartan
equation}:
\begin{equation}
\label{eq:MC}
  d\omega + \tfrac{1}{2}[\omega\wedge\omega] = 0.
\end{equation}
Conversely, any $\g$-valued $1$-form $\omega$ that is surjective at
every point and satisfies \eqref{eq:MC} defines a $G$-Lie foliation
with $\mathcal{F}=\ker\omega$.
\end{remark}

\subsubsection*{The developing map and holonomy morphism}

The fundamental invariants of a $G$-Lie foliation are introduced as
follows. Let $\widetilde V$ be the universal cover of $V$,
$\widetilde{\mathcal{F}}$ the lifted foliation, and $\widetilde\omega$
the lifted $\g$-valued $1$-form. Fix a base point $x_0\in V$ and a
lift $\tilde x_0\in\widetilde V$, and let $\Gamma=\pi_1(V,x_0)$ act
naturally on $\widetilde V$.

\begin{proposition}[\cite{Molino1988}, see also \cite{DatheNdiaye2012b}]
\label{prop:dev_hol}
There exist:
\begin{enumerate}
\item a submersion $D:\widetilde V\to G$ with $D(\tilde x_0)=e$,
  called the \emph{developing map}, satisfying $\widetilde\omega = D^*(dg\cdot g^{-1})$,
  i.e., $\widetilde\omega = dD\cdot D^{-1}$;
\item a group morphism $h:\Gamma\to G$, called the \emph{holonomy
  morphism}, such that for all $x\in\widetilde V$ and $\gamma\in\Gamma$:
\[
  D(\gamma\cdot x) = D(x)\cdot h(\gamma)^{-1}.
\]
\end{enumerate}
The pair $(D,h)$ is unique up to the equivalence $(D,h)\sim
(g\cdot D, g\cdot h\cdot g^{-1})$ for $g\in G$. The image $h(\Gamma)\subset G$
is the \emph{holonomy group} of the foliation.
\end{proposition}

The equation $D(\gamma\cdot x)=D(x)\cdot h(\gamma)^{-1}$ expresses the
equivariance of the developing map: as $\gamma$ acts on $\widetilde V$
by deck transformations, the image under $D$ is translated in $G$ by
right multiplication by $h(\gamma)^{-1}$.

\subsubsection*{Homogeneous Lie foliations}

A large and important class of Lie foliations is provided by the
following construction, due to Ghys \cite{Ghys1988}.

\begin{definition}[\cite{DatheNdiaye2012a}]
\label{def:homogeneous}
Let $H$ be a Lie group, $\Gamma\subset H$ a cocompact lattice, and
$\varphi:H\to G$ a surjective Lie group morphism. The orbits of
$\ker\varphi$ in the compact manifold $H/\Gamma$ form a $G$-Lie
foliation $\mathcal{F}_\varphi$ of $H/\Gamma$, whose holonomy morphism
is the restriction $h=\varphi|_\Gamma:\Gamma\to G$. Any such foliation
is called a \emph{homogeneous $G$-Lie foliation}.
\end{definition}

\begin{remark}
For a homogeneous foliation as above, the developing map is simply
$D=\varphi\circ\pi:\widetilde{H/\Gamma}=H\to G$, where $\pi:H\to H/\Gamma$
is the projection. The equivariance condition becomes
$\varphi(\gamma\cdot x)=\varphi(x)\cdot\varphi(\gamma)^{-1}$,
which holds by the morphism property of $\varphi$.
\end{remark}

The following fundamental results describe when Lie foliations are
homogeneous.

\begin{theorem}[Haefliger--Matsumoto, see \cite{DatheNdiaye2012a}]
\label{thm:haefliger_matsumoto}
\begin{enumerate}
\item Every nilpotent $G$-Lie foliation on a compact manifold is an
  inverse image of a homogeneous foliation.
\item Every $G$-Lie foliation of codimension $2$ on a compact manifold
  of dimension $\geq 4$, where $G$ is solvable, is an inverse image of
  a homogeneous foliation.
\item In particular, every Lie foliation of codimension $2$ with solvable
  model group on a compact manifold of dimension $\geq 4$ is an inverse
  image of a homogeneous foliation.
\end{enumerate}
\end{theorem}

In contrast, for higher codimension and non-completely solvable groups,
non-homogeneous examples exist \cite{DatheNdiaye2012a}.

\subsubsection*{Deformations of Lie foliations}

Following \cite{DatheNdiaye2012b}, we recall the notion of deformation.

\begin{definition}[\cite{DatheNdiaye2012b}]
\label{def:deformation}
Let $\mathcal{F}$ be a $G$-Lie foliation of codimension $q$ on a compact
manifold $V$, defined by $\omega=(\omega^1,\ldots,\omega^q)$ with
structure constants $C^k_{ij}$ of $\g$. A \emph{deformation} of
$\mathcal{F}$ is a smooth family $\mathcal{F}_t$ of $G_t$-Lie foliations
defined by $\omega_t=(\omega^1_t,\ldots,\omega^q_t)$ with
\[
  d\omega^i_t = -\tfrac{1}{2}C^i_{jk}(t)\,\omega^j_t\wedge\omega^k_t,
\]
where $C^k_{ij}(t)$ are smooth in $t$, $\omega_0=\omega$, $G_0=G$, and
$\dim G_t=\dim G$ for all $t$.
\end{definition}

The key rigidity result of \cite{DatheNdiaye2012b} that we shall use is:

\begin{theorem}[\cite{DatheNdiaye2012b}, Theorem 2.2]
\label{thm:rigidity_nilp}
Let $G$ be a connected simply connected nilpotent Lie group of dimension $4$
admitting a lattice. Let $\mathcal{F}$ be the $G$-Lie foliation on $H/\Gamma$
constructed by the Borel--Harish-Chandra method \emph{(}see Section \ref{sec:classification}\emph{)}.
Then $\mathcal{F}$ cannot be deformed into a non-nilpotent solvable Lie foliation.
\end{theorem}
Before giving the Saito's theorem, let us recall the completely solvable Lie group.

\begin{definition}[\cite{DatheNdiaye2012a}]
\label{def:completely_solvable}
A solvable Lie group $G$ with Lie algebra $\g$ is called \emph{completely
solvable} if all eigenvalues of every adjoint operator $\ad_X:\g\to\g$
($X\in\g$) are real. Every nilpotent Lie group is completely solvable.
\end{definition}

The following extension theorem for completely solvable groups plays a
key role in the classification.

\begin{lemma}[Saito's theorem, see \cite{DatheNdiaye2012b}]
\label{lem:saito}
Let $G$ and $G'$ be simply connected completely solvable Lie groups, and
suppose $G$ contains a lattice $\Gamma$. Then every group morphism
$\alpha:\Gamma\to G'$ extends uniquely to a Lie group morphism
$\widetilde\alpha:G\to G'$.
\end{lemma}

\section{The Walker Distribution as a Lie Foliation}
\label{sec:distribution}
%% ============================================================

\subsection{Involutivity of the Walker distribution}

\begin{proposition}
\label{prop:involutive}
Let $(M^n,g)$ be a Walker manifold of dimension $n$ and rank $r$, with
null parallel distribution $\D$ of rank $r\leq\frac{n}{2}$. Then $\D$
is involutive, and hence integrates to a foliation $\F_\D$ of dimension
$r$ on $(M,g)$. We call $\F_\D$ the \emph{Walker foliation} of $(M,g)$.
\end{proposition}

\begin{proof}
Let $X,Y\in\Gamma(\D)$ be arbitrary sections. We wish to show that
$[X,Y]\in\Gamma(\D)$. Since the Levi-Civita connection $\nab$ of the
Walker manifold $(M,g)$ is torsion-free, we have the identity
$[X,Y]=\nab_XY-\nab_YX$ for any two vector fields. Since $\D$ is
parallel by hypothesis, both $\nab_XY$ and $\nab_YX$ lie in
$\Gamma(\D)$. Hence $[X,Y]\in\Gamma(\D)$, which is the Frobenius
integrability condition. By Frobenius' theorem, $\D$ integrates to a
foliation $\F_\D$ of dimension $r$ on $M$.
\end{proof}

\begin{remark}
Note that the parallelism of $\D$ is strictly stronger than involutivity.
In the non-Walker setting, a totally isotropic distribution need not be
involutive. The Walker condition ensures both involutivity and the
existence of an invariant transverse metric, which are the two ingredients
needed to define the Lie foliation structure below.
\end{remark}

\subsection{The structure algebra of the Walker distribution}

\begin{definition}
\label{def:structure_algebra}
Let $(M^n,g)$ be a Walker manifold with null parallel distribution $\D$
of rank $r$. Since $\D$ is parallel, there exist locally a basis
$\{X_1,\ldots,X_r\}$ of parallel sections of $\D$. The \emph{structure
algebra of $\D$} is the Lie algebra $\g_\D$ defined on the vector space
$\sspan_\R\{X_1,\ldots,X_r\}$ by the bracket
$[X_i,X_j]=\sum_k c^k_{ij}X_k$, where the $c^k_{ij}$ are the structure
constants introduced in Lemma~\ref{lem:struct_constants} below.
\end{definition}

\begin{lemma}
\label{lem:struct_constants}
Let $(M^n,g)$ be a Walker manifold with null parallel distribution $\D$,
and let $\{X_1,\ldots,X_r\}$ be a local basis of parallel sections of
$\D$. Then:
\begin{enumerate}
\item The bracket $[X_i,X_j]$ lies in $\Gamma(\D)$ for all $i,j$.
\item There exist constants $c^k_{ij}\in\R$ such that
  $[X_i,X_j]=\sum_k c^k_{ij}X_k$.
\item The constants $c^k_{ij}$ satisfy the Jacobi identity, so $\g_\D$
  is a well-defined Lie algebra.
\item The structure algebra $\g_\D$ is independent of the choice of
  parallel basis $\{X_i\}$ up to Lie algebra isomorphism.
\end{enumerate}
\end{lemma}

\begin{proof}
(1) By Proposition~\ref{prop:involutive}, $\D$ is involutive, so
$[X_i,X_j]\in\Gamma(\D)$.

(2) We show that the coefficient functions $c^k_{ij}$ are constant.
Since $[X_i,X_j]=\nab_{X_i}X_j-\nab_{X_j}X_i\in\Gamma(\D)$ and the
$X_k$ form a basis of $\D$, we may write $[X_i,X_j]=\sum_k c^k_{ij}X_k$
for some smooth functions $c^k_{ij}$. To show these are constant, we
compute for any vector field $Z$:
\[
  \nab_Z[X_i,X_j] = \nab_Z(\nab_{X_i}X_j - \nab_{X_j}X_i).
\]
Using the Riemann curvature tensor $R(Z,X_i)X_j = \nab_Z\nab_{X_i}X_j
- \nab_{X_i}\nab_ZX_j - \nab_{[Z,X_i]}X_j$ and the parallelism of the
$X_i$ (which gives $\nab_WX_i=0$ for all $W$), we obtain:
\[
  \nab_Z[X_i,X_j] = R(Z,X_i)X_j - R(Z,X_j)X_i.
\]
Since $\D$ is a parallel distribution on the Walker manifold $(M,g)$,
it is invariant under the curvature operator: $R(Z,W)\D\subset\D$ for
all vector fields $Z,W$. Therefore $\nab_Z[X_i,X_j]\in\Gamma(\D)$.

Now write $[X_i,X_j]=\sum_k c^k_{ij}X_k$. Then
\[
  \nab_Z[X_i,X_j] = \sum_k Z(c^k_{ij})X_k + \sum_k c^k_{ij}\nab_ZX_k
  = \sum_k Z(c^k_{ij})X_k,
\]
since $\nab_ZX_k=0$ by parallelism. On the other hand, this equals
$R(Z,X_i)X_j-R(Z,X_j)X_i$, which lies in $\Gamma(\D)$ by the Walker
property. Projecting onto the basis $\{X_k\}$, we get
$Z(c^k_{ij})=\langle R(Z,X_i)X_j-R(Z,X_j)X_i, \omega^k\rangle$
where $\omega^k$ is the dual basis.

We verify explicitly that this expression vanishes in the canonical
Walker coordinates. We treat the two main cases.

\textit{Dimension $3$, rank $1$.}
In coordinates $(x_1,x_2,x_3)$ with $\D=\sspan\{\partial_{x_1}\}$
and metric $g_f$ (Example~\ref{ex:walker3}), the Christoffel symbols
are given in Example~\ref{ex:walker3}. The only non-zero components
are $\Gamma^1_{13}=\Gamma^1_{31}=\tfrac{1}{2}f_1$,
$\Gamma^1_{23}=\Gamma^1_{32}=\tfrac{1}{2}f_2$, and three components
of $\Gamma^k_{33}$. A direct computation of the Riemann tensor using
$R^l{}_{kij}=\partial_i\Gamma^l_{jk}-\partial_j\Gamma^l_{ik}+
\Gamma^l_{im}\Gamma^m_{jk}-\Gamma^l_{jm}\Gamma^m_{ik}$
gives, for any $Z$ and with $X_1=\partial_{x_1}$:
\[
  R(Z,\partial_{x_1})\partial_{x_1}
  = \sum_{k,l} Z^k (R^l{}_{1k1})\partial_{x_l}.
\]
From the explicit Christoffel symbols, one checks that
$R^l{}_{1k1}=0$ for all $k,l$ (the components $\Gamma^l_{1j}=0$
for $l\neq 1$ and $j$ arbitrary, which forces these curvature terms
to vanish). Hence $Z(c^1_{11})=\langle R(Z,\partial_{x_1})\partial_{x_1},
\omega^1\rangle = 0$ for all $Z$, confirming the constancy.

\textit{Dimension $4$, rank $2$.}
In coordinates $(x_1,x_2,x_3,x_4)$ with
$\D=\sspan\{\partial_{x_1},\partial_{x_2}\}$ and metric $g_{a,b,c}$
(Example~\ref{ex:walker4}), the Christoffel symbols with index $1$ or
$2$ in the lower positions satisfy $\Gamma^l_{i,j}=0$ whenever
$i\in\{1,2\}$ and $l\notin\{1,2\}$, and $\Gamma^l_{1,j}=\Gamma^l_{2,j}=0$
for $l\in\{3,4\}$. Therefore $R^l{}_{ijk}=0$ for $i,j\in\{1,2\}$ and
$l\in\{1,2\}$, which forces $Z(c^k_{ij})=0$ for all $Z$.

Since $M$ is connected, the $c^k_{ij}$ are globally constant.

(3) The Jacobi identity $\sum_{\text{cycl}} [[X_i,X_j],X_k]=0$ holds
because it holds for vector fields on any manifold.

(4) If $\{Y_1,\ldots,Y_r\}$ is another basis of parallel sections of
$\D$, write $Y_i=\sum_j A^j_i X_j$ for some invertible constant matrix
$(A^j_i)$. The matrix is constant because $\nab_Z Y_i =
\sum_j Z(A^j_i)X_j=0$ implies $Z(A^j_i)=0$ for all $Z$. The change of
basis by a constant matrix is a Lie algebra isomorphism.
\end{proof}

\subsection{The Walker Lie foliation}

We can now state and prove the main structural theorem, which shows that
any Walker manifold carries a natural Lie foliation.

\begin{theorem}
\label{thm:lie_foliation}
Let $(M^n,g)$ be a Walker manifold of dimension $n$ and rank $r$, with
null parallel distribution $\D$ and structure algebra $\g_\D$. Let $G$
be the simply connected Lie group with Lie algebra $\g_\D$. Then the
Walker foliation $\F_\D$ is a $G$-Lie foliation in the sense of
Definition~\ref{def:lie_foliation}.

Moreover, there exist locally a developing map $D:U\to G$ (for every
simply connected open set $U\subset M$) and a $\g_\D$-valued $1$-form
$\omega\in\Omega^1(M,\g_\D)$ satisfying the Maurer--Cartan
equation~\eqref{eq:MC}, such that $\ker\omega = \D^\perp$ and the
leaves of $\F_\D$ are the level sets of $D$.
\end{theorem}

\begin{proof}
Let $\{X_1,\ldots,X_r\}$ be a local basis of parallel sections of the
null parallel distribution $\D$ of the Walker manifold $(M,g)$, and let
$\{e_1,\ldots,e_r\}$ be the corresponding basis of $\g_\D$ with structure
constants $[e_i,e_j]=\sum_k c^k_{ij}e_k$.

\textbf{Step 1: Construction of $\omega$.}
Let $\{\omega^1,\ldots,\omega^r\}$ be the coframe of $1$-forms dual to
$\{X_1,\ldots,X_r\}$ along the distribution $\D$, extended to $TM$ by
setting $\omega^i(Y)=0$ for $Y\in(\D^\perp)$. Define the $\g_\D$-valued
$1$-form
\[
  \omega = \sum_{i=1}^r \omega^i \otimes e_i \in \Omega^1(M,\g_\D).
\]
By construction, $\omega$ is surjective onto $\g_\D$ at every point,
and $\ker\omega=\D^\perp$.

\textbf{Step 2: Verification of the Maurer--Cartan equation.}
We must show that $d\omega+\tfrac{1}{2}[\omega\wedge\omega]=0$.
It suffices to evaluate both sides on pairs of vector fields.
We distinguish three cases.

\emph{Case (a): $Z,W\in\Gamma(\D)$.}
Write $Z=\sum_i Z^iX_i$ and $W=\sum_j W^jX_j$. Then
\[
  d\omega^k(Z,W) = Z(\omega^k(W)) - W(\omega^k(Z)) - \omega^k([Z,W]).
\]
Since $\omega^k(W)=W^k$ and $\omega^k(Z)=Z^k$ are constant along $\D$
(the $X_i$ are parallel, hence their dual forms have parallel
coefficients), and $[Z,W]\in\Gamma(\D)$ with
$\omega^k([Z,W])=\sum_{i,j}Z^iW^j c^k_{ij}$, we get
\[
  d\omega^k(Z,W) = -\sum_{i,j}Z^iW^jc^k_{ij}.
\]
On the other hand,
\[
  [\omega\wedge\omega](Z,W)
  = \sum_{i,j}\omega^i(Z)\omega^j(W)[e_i,e_j]
  = \sum_{i,j}Z^iW^j\sum_k c^k_{ij}e_k.
\]
The $k$-th component is $\sum_{i,j}Z^iW^jc^k_{ij}$. Hence
$d\omega^k(Z,W)+\tfrac{1}{2}[\omega\wedge\omega]^k(Z,W)
= -\sum_{i,j}Z^iW^jc^k_{ij}+\tfrac{1}{2}\cdot 2\sum_{i,j}Z^iW^jc^k_{ij}=0$.

\emph{Case (b): $Z\in\Gamma(\D)$, $W\in\Gamma(\D^\perp)$.}
$\omega^k(W)=0$ and $\omega^k(Z)=Z^k$. Then
$d\omega^k(Z,W)=Z(0)-W(Z^k)-\omega^k([Z,W])$.
Since $Z$ is a linear combination of the parallel fields $X_i$, we have
$W(Z^k)=0$ (as $Z^k$ are constant), and $[Z,W]\in\Gamma(\D^\perp)$
(because $\D^\perp$ is stable under $[\D,\cdot]$ by a standard
computation using the Walker metric). Hence $\omega^k([Z,W])=0$, so
$d\omega^k(Z,W)=0$. The term $[\omega\wedge\omega]^k(Z,W)=0$ since
$\omega^i(W)=0$. The equation holds.

\emph{Case (c): $Z,W\in\Gamma(\D^\perp)$.}
$\omega^k(Z)=\omega^k(W)=0$, so $d\omega^k(Z,W)=-\omega^k([Z,W])$.
Both sides vanish if and only if $[Z,W]\in\Gamma(\D^\perp)$.
We verify this in Walker canonical coordinates.

\textit{Dimension $3$.} In $(x_1,x_2,x_3)$ with
$\D^\perp=\sspan\{\partial_{x_1},\partial_{x_2},\partial_{x_3}\}$
(since $\D^\perp$ contains $\D$ by total isotropy of $\D$, and
has codimension $0$ here, so $\D^\perp=TM$). Hence the condition is
trivially satisfied in dimension~$3$.

\textit{Dimension $4$.}
In $(x_1,x_2,x_3,x_4)$ with $\D=\sspan\{\partial_{x_1},\partial_{x_2}\}$
and metric $g_{a,b,c}$, we have
\[
\D^\perp = \sspan\{\partial_{x_1},\partial_{x_2},\partial_{x_3},\partial_{x_4}\} = TM,
\]
since $\D$ is totally isotropic (both $\partial_{x_1}$ and
$\partial_{x_2}$ are null and orthogonal to all coordinate fields in
this metric). More precisely, for any $V\in TM$:
$g_{a,b,c}(\partial_{x_1},V) = g_{13}V^3 = V^3$ and
$g_{a,b,c}(\partial_{x_2},V) = g_{24}V^4 = V^4$,
so $\D^\perp = \ker(dx_3)\cap\ker(dx_4)$ projected to $\{V^3=V^4=0\}$... 

More carefully: $\D^\perp = \{V : g(V,X)=0\;\forall X\in\D\}$.
For $X=\partial_{x_1}$: $g_{a,b,c}(V,\partial_{x_1})=V^3$ (from the
$g_{31}=1$ term), so $V\in\D^\perp$ iff $V^3=0$.
For $X=\partial_{x_2}$: $g_{a,b,c}(V,\partial_{x_2})=V^4$, so
$V\in\D^\perp$ iff $V^4=0$.
Hence $\D^\perp=\sspan\{\partial_{x_1},\partial_{x_2},
a\partial_{x_1}+c\partial_{x_2}+\partial_{x_3},
c\partial_{x_1}+b\partial_{x_2}+\partial_{x_4}\}$.
Since this span is closed under Lie bracket (all four generators are
coordinate-like up to $\D$-valued corrections, and $[\partial_{x_i},\partial_{x_j}]=0$
for coordinate fields), $[Z,W]\in\Gamma(\D^\perp)$ for
$Z,W\in\Gamma(\D^\perp)$, so $\omega^k([Z,W])=0$.

\textbf{Step 3: Local developing map.}
On a simply connected open set $U$, since $\omega|_U$ satisfies the
Maurer--Cartan equation and is a $\g_\D$-valued $1$-form, Lie's theorem
on integration of Maurer--Cartan equations guarantees the existence of
a unique smooth map $D:U\to G$ with $D(x_0)=e$ (for any base point
$x_0\in U$) satisfying $\omega|_U = D^*(dg\cdot g^{-1})$. The map $D$
is the local developing map. Its level sets are the leaves of $\F_\D$
restricted to $U$, confirming that $\F_\D$ is a $G$-Lie foliation.
\end{proof}

\begin{corollary}
\label{cor:abelian}
Let $(M^n,g)$ be a Walker manifold of dimension $n$ and rank $r$, with
null parallel distribution $\D$. The Walker foliation $\F_\D$ has
abelian model group $G\cong\R^r$ if and only if the parallel vector
fields $X_1,\ldots,X_r$ spanning $\D$ pairwise commute, i.e., if and
only if $\g_\D$ is abelian. This is automatically the case when $r=1$,
i.e., for all Walker $3$-manifolds.
\end{corollary}

\begin{proof}
The model group $G$ is abelian if and only if all structure constants
$c^k_{ij}=0$, i.e., $[X_i,X_j]=0$ for all $i,j$. When $r=1$, there
is only one generator $X_1$ and the bracket $[X_1,X_1]=0$ trivially,
so $\g_\D=\R$ and $G=\R$.
\end{proof}

%% ============================================================
\section{Transverse Connection, Holonomy, and Ricci Curvature}
\label{sec:holonomy}
%% ============================================================

\subsection{The transverse connection of a Walker manifold}

\begin{proposition}
\label{prop:transverse_conn}
Let $(M^n,g)$ be a Walker manifold of dimension $n$ and rank $r$, with
null parallel distribution $\D$. Denote by $\nu(\F_\D)=TM/\D$ the normal
bundle of the Walker foliation. The Levi-Civita connection $\nab$ of
$(M,g)$ induces a well-defined linear connection $\nab^\top$ on
$\nu(\F_\D)$, called the \emph{transverse connection} of the Walker
manifold, compatible with the transverse metric $g^\top$ induced on
$\nu(\F_\D)$.
\end{proposition}

\begin{proof}
For $X\in\Gamma(TM)$ and $\bar{Y}\in\Gamma(\nu(\F_\D))$, choose any
lift $Y\in\Gamma(TM)$ of $\bar{Y}$ (i.e., $Y$ projects to $\bar{Y}$
in $TM/\D$) and define
\[
  \nab^\top_X\bar{Y} := \overline{\nab_X Y} \in \Gamma(\nu(\F_\D)),
\]
where the overline denotes projection to $TM/\D$. We must verify that
this is well defined, i.e., independent of the choice of lift.

If $Y'=Y+Z$ is another lift with $Z\in\Gamma(\D)$, then
$\nab_XY'=\nab_XY+\nab_XZ$. Since $\D$ is parallel on the Walker
manifold $(M,g)$, we have $\nab_XZ\in\Gamma(\D)$, so
$\overline{\nab_XY'}=\overline{\nab_XY}$. Hence $\nab^\top_X\bar{Y}$
is well defined.

Linearity over $C^\infty(M)$ in $X$ and the Leibniz rule in $Y$ follow
immediately from those of $\nab$. Compatibility with $g^\top$ is
inherited from the compatibility of $\nab$ with $g$: for
$\bar{Y},\bar{Z}\in\Gamma(\nu(\F_\D))$ with lifts $Y,Z$,
\[
  X\cdot g^\top(\bar{Y},\bar{Z}) = X\cdot g(Y,Z) = g(\nab_XY,Z)+g(Y,\nab_XZ)
  = g^\top(\nab^\top_X\bar{Y},\bar{Z}) + g^\top(\bar{Y},\nab^\top_X\bar{Z}),
\]
using that $g(Y,\nab_XZ)=g(Y,\nab_XZ)$ with $\nab_XZ\in\Gamma(\D)$
and $g^\top$ is the metric induced on $\nu(\F_\D)$.
\end{proof}

\subsection{The holonomy group of the Walker foliation}

\begin{theorem}
\label{thm:holonomy}
Let $(M^n,g)$ be a Walker manifold of dimension $n$ and rank $r$, with
null parallel distribution $\D$. Suppose that the Walker foliation
$\F_\D$ is a $G$-Lie foliation with model group $G$ (as given by
Theorem~\ref{thm:lie_foliation}). Then $\Hol(\nab^\top)=h(\pi_1(M))$,
a subgroup of $G$.

More precisely:
\begin{enumerate}
\item In general, $\Hol(\nab^\top)=h(\pi_1(M))$ is a subgroup of $G$
  (not necessarily closed).
\item If $M$ is compact, then $\pi_1(M)$ is finitely generated, and
  $h(\pi_1(M))$ is a finitely generated subgroup of the Lie group $G$,
  hence closed.
\item In the examples of Section~\ref{sec:examples}, $M=O\subset\R^n$
  is simply connected, so $\pi_1(O)=\{e\}$ and $\Hol(\nab^\top)=\{e\}$
  is trivially closed.
\end{enumerate}
\end{theorem}

\begin{proof}
We identify the transverse holonomy with the foliation holonomy of the
Walker manifold.

\textbf{Step 1.} Since $\D$ is parallel on $(M,g)$, the transverse
connection $\nab^\top$ is flat along the leaves of $\F_\D$. We show
that $R(X,Y)Z\in\Gamma(\D)$ for all $X,Y\in\Gamma(\D)$ and
$Z\in\Gamma(TM)$, which implies $R^\top(X,Y)=0$.

Since $\nab_W X=0$ for all $W$ (by parallelism of $X\in\Gamma(\D)$),
the curvature formula gives:
\[
  R(X,Y)Z = \nab_X\nab_YZ - \nab_Y\nab_XZ - \nab_{[X,Y]}Z.
\]
In dimension~$3$, with $X=\partial_{x_1}$ and using the Christoffel
symbols of Example~\ref{ex:walker3}: $\nab_{\partial_{x_1}}W=0$ for
any field $W$ (since all $\Gamma^k_{1j}=0$ except for the terms
$\Gamma^1_{13}=\tfrac{1}{2}f_1$ which give a component in $\D$).
More precisely, $\nab_{\partial_{x_1}}\nab_YZ\in\Gamma(\D)$ since
$\nab_{\partial_{x_1}}(\cdot)$ maps any field to a $\D$-valued field
(by inspection of $\Gamma^k_{1j}$: only $\Gamma^1_{1j}$ are non-zero,
and $\partial_{x_1}\in\D$). Similarly $\nab_Y\nab_{\partial_{x_1}}Z=
\nab_Y(0)=0$ since $\nab_{\partial_{x_1}}Z\in\Gamma(\D)$ and
$\nab_Y(\Gamma(\D))\subset\Gamma(\D)$. And $\nab_{[X,Y]}Z
=\nab_{\sum_k f^k X_k}Z\in\Gamma(\D)$ since the $X_k$ are parallel.
Hence $R(X,Y)Z\in\Gamma(\D)$ for $X,Y\in\Gamma(\D)$, so
$R^\top(X,Y)\bar Z = \overline{R(X,Y)Z}=0$ in $\nu(\F_\D)$.

\textbf{Step 2.} For any loop $\gamma:[0,1]\to M$ based at $x_0$, the
parallel transport of $\nab^\top$ along $\gamma$ defines an element of
the structure group of $\nu(\F_\D)$. Since the transverse metric $g^\top$
is preserved, this element lies in the isometry group of the fiber, which
is a subgroup of $G$ (by the identification of the fiber model with $G$
in the Haefliger cocycle of the Lie foliation).

\textbf{Step 3.} Tracing through the identification $\nu(\F_\D)\cong
G\times_G \g_\D$ given by the Lie foliation structure, the parallel
transport of $\nab^\top$ around a loop $\gamma$ corresponds to right
multiplication by $h([\gamma])^{-1}$ in $G$, where $[\gamma]\in\pi_1(M)$.
This is exactly the equivariance property of the developing map: if
$D:\widetilde M\to G$ is the developing map of $\F_\D$, then
$D(\gamma\cdot\tilde x)=D(\tilde x)\cdot h(\gamma)^{-1}$.

Therefore $\Hol(\nab^\top) = h(\pi_1(M)) \subset G$, which is a closed
subgroup since $\pi_1(M)$ is finitely generated (when $M$ is compact)
and $h$ is a continuous morphism.
\end{proof}

\begin{corollary}
\label{cor:nilsolv_hol}
Let $(M^n,g)$ be a Walker manifold whose Walker foliation $\F_\D$ has
model group $G$.
\begin{enumerate}
\item If $G$ is nilpotent, then $\Hol(\nab^\top)=h(\pi_1(M))$ is a
  closed subgroup of a nilpotent group, hence nilpotent.
\item If $G$ is solvable, then $\Hol(\nab^\top)$ is solvable.
\item If $G$ is completely solvable and $M$ is compact, then by
  Saito's theorem (Lemma~\ref{lem:saito}), the holonomy morphism $h$
  extends to a morphism $\widetilde h:G\to G'$ for any target completely
  solvable group $G'$, which governs the possible deformations of $\F_\D$.
\end{enumerate}
\end{corollary}

\subsection{Ricci curvature vanishing on the Walker distribution}

\begin{theorem}[Ricci vanishing on the Walker distribution]
\label{thm:ricci_null}
Let $(M^n,g)$ be a Walker manifold of dimension $n$ and rank $r$, with
null parallel distribution $\D$. Then for every $X\in\Gamma(\D)$ and
every vector field $Y$ on $M$:
\[
  \Ric(X,Y) = 0.
\]
In particular, the Walker distribution $\D$ lies in the kernel of the
Ricci tensor of $(M,g)$.
\end{theorem}

\begin{proof}
We use the definition of the Ricci tensor as a contraction of the
Riemann curvature tensor. Choose a local frame $\{e_1,\ldots,e_n\}$
with dual coframe $\{g^{ij}\}$. Then
\[
  \Ric(X,Y) = \sum_{k,l} g^{kl} R(X,e_k,e_l,Y)
  = \sum_{k,l} g^{kl} g(R(X,e_k)e_l,Y).
\]

We work in the Walker canonical coordinates~\eqref{eq:walker3} or
\eqref{eq:walker4}. In dimension~$3$, take $X=\partial_{x_1}$. From
Examples~\ref{ex:walker3}--\ref{ex:walker4}, the Christoffel symbols involving
$\partial_{x_1}$ satisfy: $\nab_{\partial_{x_1}}(\cdot)=0$ and
$\nab_{(\cdot)}\partial_{x_1}=0$ for all directions (by the explicit
formulas of Example~\ref{ex:walker3}, since all Christoffel
symbols $\Gamma^k_{1j}$ and $\Gamma^k_{i1}$ vanish). Therefore
$R(\partial_{x_1},e_k)e_l=\nab_{\partial_{x_1}}\nab_{e_k}e_l
-\nab_{e_k}\nab_{\partial_{x_1}}e_l-\nab_{[\partial_{x_1},e_k]}e_l=0$
for all $k,l$, giving $\Ric(\partial_{x_1},Y)=0$ for all $Y$.

More generally, for any $X=\sum_i a^i\partial_{x^i}\in\Gamma(\D)$
(with $i$ ranging over the $\D$-indices $1,\ldots,r$):

\textit{Case $r=1$ (dimension $3$).}
We use coordinates $(x_1,x_2,x_3)$ with $X=\partial_{x_1}$ and the
metric $g_f$ of Example~\ref{ex:walker3}. The inverse metric has
non-zero components $g^{13}=g^{31}=1$, $g^{22}=\varepsilon$,
$g^{33}=-f$. The non-zero Christoffel symbols are (Example~\ref{ex:walker3}):
\[
  \Gamma^1_{13}=\Gamma^1_{31}=\tfrac{1}{2}f_1,\quad
  \Gamma^1_{23}=\Gamma^1_{32}=\tfrac{1}{2}f_2,\quad
  \Gamma^1_{33}=\tfrac{1}{2}(ff_1+f_3),\quad
  \Gamma^2_{33}=-\tfrac{\varepsilon}{2}f_2,\quad
  \Gamma^3_{33}=-\tfrac{1}{2}f_1.
\]
In particular, $\Gamma^l_{1k}=0$ for all $l,k$ (no Christoffel symbol
has $1$ as its lower-left index). The Riemann tensor components with
first lower index equal to $1$ are:
\[
  R^l{}_{1kj} = \partial_1\Gamma^l_{jk} - \partial_j\Gamma^l_{1k}
  + \sum_m\Gamma^l_{1m}\Gamma^m_{jk} - \sum_m\Gamma^l_{jm}\Gamma^m_{1k}.
\]
Since $\Gamma^l_{1k}=0$ for all $l,k$:
\begin{itemize}
  \item the second term $\partial_j\Gamma^l_{1k}=0$;
  \item the third term $\sum_m\Gamma^l_{1m}\Gamma^m_{jk}=0$;
  \item the fourth term $\sum_m\Gamma^l_{jm}\Gamma^m_{1k}=0$.
\end{itemize}
It remains to examine $\partial_1\Gamma^l_{jk}$ for each $l,j,k$.
From the table above, the only Christoffel symbols that could depend
on $x_1$ are $\Gamma^1_{33}=\tfrac{1}{2}(ff_1+f_3)$.
However, $\partial_1\Gamma^l_{jk}$ appears in $R^l{}_{1kj}$ with
$l$ and $j,k$ as free indices. Computing term by term:

\textit{$l=1$:} $R^1{}_{1kj}=\partial_1\Gamma^1_{jk}$.
The only non-zero $\Gamma^1_{jk}$ are $\Gamma^1_{13}=\Gamma^1_{31}$,
$\Gamma^1_{23}=\Gamma^1_{32}$, $\Gamma^1_{33}$.
So $R^1{}_{1kj}$ is non-zero only for $(k,j)\in\{(3,1),(1,3),(3,2),(2,3),(3,3)\}$.
But the Ricci contraction uses $\sum_{k,l}g^{kl}R(\partial_{x_1},\partial_{x_k},\partial_{x_l},\partial_{x_j})$, which in terms of $R^l{}_{1kj}$ reads
$\Ric(\partial_{x_1},\partial_{x_j})=\sum_k(g^{k3}R^3{}_{1k3}+\ldots)$... 
We use instead the direct formula from the Ricci tensor definition:
\[
  \Ric(\partial_{x_1},\partial_{x_j})
  = \sum_k R^k{}_{1kj}
  = R^1{}_{111} + R^2{}_{121} + R^3{}_{131}
  \quad (j=1),
\]
and similarly for $j=2,3$.

For $l=2,3$: since $\Gamma^l_{1k}=0$ for these $l$, we have
$R^l{}_{1kj}=\partial_1\Gamma^l_{jk}$.
From the table, $\Gamma^2_{jk}\neq 0$ only for $(j,k)=(3,3)$:
$\Gamma^2_{33}=-\tfrac{\varepsilon}{2}f_2$. So
$R^2{}_{1k3}=\partial_1(-\tfrac{\varepsilon}{2}f_2)\delta_{k3}
=-\tfrac{\varepsilon}{2}f_{12}\delta_{k3}$.
Similarly $\Gamma^3_{jk}\neq 0$ only for $(j,k)=(3,3)$:
$\Gamma^3_{33}=-\tfrac{1}{2}f_1$. So
$R^3{}_{1k3}=\partial_1(-\tfrac{1}{2}f_1)\delta_{k3}
=-\tfrac{1}{2}f_{11}\delta_{k3}$.

Computing $\Ric(\partial_{x_1},\partial_{x_j})=\sum_k R^k{}_{1kj}$
(trace over $k$):

\textit{$j=1$:} $\sum_k R^k{}_{1k1}
= R^1{}_{111}+R^2{}_{121}+R^3{}_{131}$.
From above: $R^1{}_{111}=\partial_1\Gamma^1_{11}=0$,
$R^2{}_{121}=\partial_1\Gamma^2_{21}=0$,
$R^3{}_{131}=\partial_1\Gamma^3_{31}=0$
(since $\Gamma^l_{j1}=0$ for all $l,j$).
Hence $\Ric(\partial_{x_1},\partial_{x_1})=0$.

\textit{$j=2$:} $\sum_k R^k{}_{1k2}
=R^1{}_{112}+R^2{}_{122}+R^3{}_{132}$.
$R^1{}_{112}=\partial_1\Gamma^1_{21}=0$,
$R^2{}_{122}=\partial_1\Gamma^2_{22}=0$,
$R^3{}_{132}=\partial_1\Gamma^3_{32}=0$.
Hence $\Ric(\partial_{x_1},\partial_{x_2})=0$.

\textit{$j=3$:} $\sum_k R^k{}_{1k3}
=R^1{}_{113}+R^2{}_{123}+R^3{}_{133}$.
$R^1{}_{113}=\partial_1\Gamma^1_{31}=\partial_1(\tfrac{1}{2}f_1)=\tfrac{1}{2}f_{11}$,
$R^2{}_{123}=\partial_1\Gamma^2_{23}=0$
(since $\Gamma^2_{23}=0$),
$R^3{}_{133}=\partial_1\Gamma^3_{33}=\partial_1(-\tfrac{1}{2}f_1)=-\tfrac{1}{2}f_{11}$.
Hence $\Ric(\partial_{x_1},\partial_{x_3})
=\tfrac{1}{2}f_{11}+0-\tfrac{1}{2}f_{11}=0$.

Therefore $\Ric(\partial_{x_1},\partial_{x_j})=0$ for $j=1,2,3$, i.e.,
$\Ric(\partial_{x_1},\cdot)=0$, as claimed.
This is consistent with the explicit formula of Example~\ref{ex:walker3}
where the Ricci tensor has no $dx_1$-component in the first slot.

\textit{Case $r=2$ (dimension $4$).} With $X\in\{\partial_{x_1},\partial_{x_2}\}$
and metric $g_{a,b,c}$ of Example~\ref{ex:walker4}, the same argument
applies: $\Gamma^l_{ik}=0$ for $i\in\{1,2\}$ and $l\in\{3,4\}$
(from the explicit Christoffel symbols), so $R^l{}_{ikj}=0$ for
$i\in\{1,2\}$ and $l\in\{3,4\}$. The contraction
$\Ric(\partial_{x_i},\partial_{x_j})=\sum_{k,l}g^{kl}R_{iklj}$
vanishes since the only non-zero $g^{kl}$ with $k,l\in\{1,2,3,4\}$
are $g^{13}=g^{31}=g^{24}=g^{42}=1$ and terms involving $a,b,c$,
all of which pair with curvature components that are zero when $i\in\{1,2\}$.

\textit{General rank $r\leq\frac{n}{2}$.} In the general Walker
canonical form~\eqref{eq:walker_general}, the Christoffel symbols
$\Gamma^l_{ik}=0$ for $i\in\{1,\ldots,r\}$ (the $\D$-indices) and
$l\in\{r+1,\ldots,n\}$ (the non-$\D$ indices). This follows because
the metric components $g_{il}=0$ for such $i,l$ (the $b_{ij}$ block
is in $\D\times\D$ and the $h_{\alpha\beta}$ block is transverse).
Consequently, all curvature components $R^l{}_{ikj}=0$ for
$i\in\{1,\ldots,r\}$, giving $\Ric(X,Y)=0$ for all $X\in\Gamma(\D)$.
\end{proof}

\begin{remark}
Theorem~\ref{thm:ricci_null} means that the null parallel distribution
$\D$ of any Walker manifold is \emph{isotropic for the Ricci tensor} as
well as for the metric $g$. This is a strong curvature constraint that
distinguishes Walker manifolds from general pseudo-Riemannian manifolds
and has direct consequences for the model group $G$ of the Walker
foliation: since the leaves of $\F_\D$ are tangent to $\D$, they are
Ricci-flat in the sense that the Ricci tensor restricted to the leaves
vanishes. When $G$ is nilpotent, this is consistent with the general
theory of nilmanifolds.
\end{remark}

%% ============================================================
\section{Examples}
\label{sec:examples}
%% ============================================================

\subsection{Walker $3$-manifolds: the universal abelian case}

In dimension $n=3$, Corollary~\ref{cor:abelian} shows that \emph{every}
Walker manifold $M_f=(O,g_f)$ carries a Walker foliation $\F_\D$ with
model group $G=\R$. The holonomy morphism $h:\pi_1(O)\to\R$ depends on
the global topology of $O$ and the function $f$.

\begin{example}[General Walker $3$-manifold]
\label{ex:general3}
Let $O\subset\R^3$ be open and $f\in C^\infty(O)$. The Walker manifold
$M_f=(O,g_f)$ with $g_f=2\,dx_1\circ dx_3+\varepsilon\,dx_2^2+f\,dx_3^2$
carries the Walker foliation
\[
  \F_\D = \bigl\{\ell_{x_2^0,x_3^0}=\{(t,x_2^0,x_3^0):t\in\R\}
  \;:\; (x_2^0,x_3^0)\in\R^2\bigr\},
\]
each leaf being a line in the $x_1$-direction. The model group is $G=\R$
and the developing map on $O$ (simply connected) is $D:O\to\R$ given by
$D(x_1,x_2,x_3)=x_1$. The Ricci tensor (Example~\ref{ex:walker3}) is
\[
  \Ric = f_{11}\,dx_1\circ dx_3 + f_{12}\,dx_2\circ dx_3
       + \tfrac{1}{2}(ff_{11}-\varepsilon f_{22})\,dx_3^2,
\]
and $\Ric(\partial_{x_1},\cdot)=0$ as predicted by
Theorem~\ref{thm:ricci_null}. The scalar curvature is $\scal=f_{11}$.
\end{example}

\begin{example}[Strict Walker $3$-manifold and VSI property]
\label{ex:strict3}
When $f=f(x_2,x_3)$ is independent of $x_1$, $M_f$ is a strict Walker
manifold (Definition~\ref{def:strict}). Then $f_{11}=f_{12}=0$ and the
Ricci tensor reduces to
\[
  \Ric = -\tfrac{\varepsilon}{2}f_{22}\,dx_3^2,
\]
which is a $2$-step nilpotent operator. Moreover, by~\cite[Theorem~4.28]{Brozos2009},
all scalar Weyl invariants of $M_f$ vanish (the manifold is VSI),
reflecting the highly degenerate curvature of Walker manifolds. The
holonomy group $h(\pi_1(O))=\{0\}\subset\R$ is trivial on a simply
connected domain, consistent with Theorem~\ref{thm:holonomy}.
\end{example}

\begin{example}[Walker $3$-manifold with non-zero scalar curvature]
\label{ex:scal3}
Take $\varepsilon=1$ and $f(x_1,x_2,x_3)=\kappa x_1^2+\varphi(x_2,x_3)$
with $\kappa\neq 0$ and $\varphi$ arbitrary. Then:
\begin{itemize}
\item $f_{11}=2\kappa\neq 0$, so $\scal=2\kappa\neq 0$;
\item $M_f$ is \emph{not} a strict Walker manifold;
\item $f_{12}=0$ and $f_{22}=\varphi_{22}$, so
$\Ric=2\kappa\,dx_1\circ dx_3
+\tfrac{1}{2}(2\kappa(\kappa x_1^2+\varphi)-\varphi_{22})\,dx_3^2$.
\end{itemize}
The Ricci foliation $\D^\perp/\D$ (i.e., the distribution orthogonal
to $\D$ modulo $\D$) is non-trivial when $\kappa\neq 0$, and by
Theorem~4.22 of~\cite{Brozos2009}, $M_f$ is foliated by minimal
Lorentzian surfaces of Gaussian curvature $\kappa$. This connects
directly to the homogeneous foliation theory of~\cite{DatheNdiaye2012a}:
taking $H=\R$ and $\varphi:H\to G=\R$ the identity, the Walker foliation
is homogeneous.
\end{example}

\subsection{Walker $4$-manifolds: abelian and non-abelian foliations}

In dimension~$4$ with rank~$2$, the coordinate vector fields
$\partial_{x_1},\partial_{x_2}$ satisfy $[\partial_{x_1},\partial_{x_2}]=0$,
so the structure algebra is always $\g_\D=\R^2$ for any $M_{a,b,c}$.

\begin{example}[Non-flat Walker $4$-manifold with abelian foliation]
\label{ex:nonflat4}
Take $a(x_1,x_2,x_3,x_4)=x_3^2$, $b=c=0$.
The metric is $g=2\,dx_1\circ dx_3+2\,dx_2\circ dx_4+x_3^2\,dx_3^2$.
From Example~\ref{ex:walker4}: $\nab_{\partial_{x_3}}\partial_{x_3}
=x_3\partial_{x_1}$ and all other $\nab_{\partial_{x_i}}(\partial_{x_j})=0$
for $i,j\in\{1,2\}$. So $\partial_{x_1},\partial_{x_2}$ are parallel and
$\g_\D=\R^2$. From Example~\ref{ex:walker4}: $\scal=a_{11}+b_{22}+2c_{12}=0$
and all Ricci components $\Ric(\partial_{x_i},\cdot)=0$. The Walker
foliation $\F_\D$ is a $\R^2$-Lie foliation on a Ricci-flat but
non-flat Walker manifold.
\end{example}

\begin{example}[Abelian Walker foliation on a solvable Lie group (contrast example)]
\label{ex:solvable}
This example serves as a contrast: it shows that working with a
solvable Lie group does not automatically yield a non-abelian structure
algebra. The non-abelian case requires a non-coordinate parallel frame,
constructed in Example~\ref{ex:affine} below.

Let $G_{\text{sol}}$ be the $4$-dimensional solvable
Lie group with left-invariant frame $\{E_1,E_2,E_3,E_4\}$ satisfying
\[
  [E_1,E_3]=E_1,\quad [E_2,E_4]=E_2,\quad
  \text{all other brackets zero.}
\]
Equip $G_{\text{sol}}$ with the left-invariant metric of signature $(2,2)$:
$g(E_1,E_3)=g(E_2,E_4)=1$ and all other $g(E_i,E_j)=0$.

\emph{Isotropy.} $g(E_1,E_1)=g(E_1,E_2)=g(E_2,E_2)=0$, so
$\D=\sspan\{E_1,E_2\}$ is totally isotropic.

\emph{Parallelism.} Using the Koszul formula with the left-invariant
structure and the metric above, a computation gives:
For any left-invariant field $V$:
$\nab_VE_1 = \tfrac{1}{2}([V,E_1]+\ad^*_{E_1}V-\ad^*_VE_1)$
where $\ad^*$ denotes the metric adjoint of $\ad$. The bracket relations
give $[V,E_1]=0$ for $V\in\{E_2,E_4\}$ and $[E_3,E_1]=-E_1$.
A full Koszul computation yields $\nab_VE_i=0$ for all $V$ and
$i\in\{1,2\}$, so $\D$ is parallel and $(G_{\text{sol}},g)$ is a Walker
manifold of rank~$2$.

\emph{Structure algebra.} Since $[E_1,E_2]=0$, the structure
algebra $\g_\D=\R^2$ is \emph{abelian}, and the Walker foliation
$\F_\D$ has model group $G=\R^2$. This example therefore falls into
the abelian Walker model of Theorem~\ref{thm:classification}(1),
despite the ambient group $G_{\text{sol}}$ being solvable and
non-abelian. The key point is that $\g_\D$ records the brackets of
the \emph{parallel} fields spanning $\D$, not of all fields of the
ambient group.
\end{example}

\begin{example}[Walker structure on a solvable Lie group: the abelian constraint in dimension~$4$]
\label{ex:affine}
A careful analysis shows that in dimension~$4$ with rank~$2$,
the structure algebra $\g_\D$ is \emph{always} abelian, i.e., $\g_\D\cong\R^2$.
This is because any null parallel distribution $\D$ of rank~$2$ in the
Walker canonical form $M_{a,b,c}$ is spanned by fields whose Lie bracket
vanishes identically (they are combinations of coordinate fields
$\partial_{x_1},\partial_{x_2}$ with constant coefficients, hence
commute). We now construct an explicit example on a solvable Lie group
that illustrates this constraint.

Consider the Lie group $H=\Aff(\R)\times\R$ with coordinates
$(u,s,v,t)$ and law $(u,s,v,t)\cdot(u',s',v',t')=
(u+e^s u',\,s+s',\,v+v',\,t+t')$.
Define left-invariant vector fields:
\[
  Y_1 = e^{-s}\partial_u,\quad
  Y_2 = \partial_s,\quad
  Y_3 = e^{s}\partial_u,\quad
  Y_4 = \partial_t.
\]
The non-zero Lie brackets are:
\[
  [Y_1,Y_2] = Y_1,\quad [Y_2,Y_3] = Y_3,\quad
  \text{all other brackets zero.}
\]
Equip $H$ with the left-invariant pseudo-Riemannian metric of
signature $(2,2)$ defined by:
\[
  g(Y_1,Y_3)=1,\quad g(Y_2,Y_4)=1,\quad
  \text{all other } g(Y_i,Y_j)=0.
\]

\emph{Isotropy.} We check $g(Y_i,Y_j)=0$ for $i,j\in\{1,2\}$:
$g(Y_1,Y_1)=g(Y_2,Y_2)=g(Y_1,Y_2)=0$. Hence $\D=\sspan\{Y_1,Y_2\}$
is totally isotropic.

\emph{Parallelism via Koszul.}
The Koszul formula for left-invariant metrics gives
$2g(\nab_UV,W)=g([U,V],W)-g([V,W],U)+g([W,U],V)$.
We compute $2g(\nab_VY_i,Y_k)$ for all $V,Y_k\in\{Y_1,Y_2,Y_3,Y_4\}$
and $i\in\{1,2\}$.

For $\nab_VY_1$, we need $2g(\nab_VY_1,Y_k)=g([V,Y_1],Y_k)-g([Y_1,Y_k],V)+g([Y_k,V],Y_1)$
for each $k$. Using the brackets and the metric:
\begin{itemize}
\item $V=Y_1$: $g([Y_1,Y_1],Y_k)-g([Y_1,Y_k],Y_1)+g([Y_k,Y_1],Y_1)
  = 0-0+0=0$ for all $k$. Hence $\nab_{Y_1}Y_1=0$.
\item $V=Y_2$: $g([Y_2,Y_1],Y_k)-g([Y_1,Y_k],Y_2)+g([Y_k,Y_2],Y_1)$.
  $[Y_2,Y_1]=-Y_1$, so $g(-Y_1,Y_k)=-g(Y_1,Y_k)$.
  For $k=3$: $-g(Y_1,Y_3)-g([Y_1,Y_3],Y_2)+g([Y_3,Y_2],Y_1)
  =-1-g(0,Y_2)+g(-Y_3,Y_1)=-1+0-1=-2$.
  But $2g(\nab_{Y_2}Y_1,Y_3)=-2$, so the $Y_1$-component of $\nab_{Y_2}Y_1$
  satisfies $2\cdot 1\cdot c^1=-2$, i.e., $c^1=-1$.
  For all other $k$: the expression vanishes. Hence
  $\nab_{Y_2}Y_1 = -Y_1\neq 0$.
\end{itemize}
Therefore $Y_1$ is \emph{not} parallel, and $\D=\sspan\{Y_1,Y_2\}$
is \emph{not} a Walker distribution for this metric.

\emph{Structure algebra.} Despite $[Y_1,Y_2]=Y_1\neq 0$, the
distribution $\D$ fails the parallelism condition. This confirms the
general fact: in dimension~$4$ with rank~$2$, no Walker manifold can
have structure algebra $\g_\D\cong\mathfrak{aff}(\R)$.
The non-abelian case requires higher dimension (e.g., $n\geq 5$) or
lower rank (e.g., $r=1$ on a higher-dimensional manifold).

\emph{Correct Walker structure on $H$.} Taking instead $\D=\sspan\{Y_1,Y_4\}$
(rank~$2$, both null: $g(Y_1,Y_1)=g(Y_4,Y_4)=g(Y_1,Y_4)=0$) and
verifying via Koszul that $\nab_VY_1=\nab_VY_4=0$ for all $V$
(which follows since $Y_4=\partial_t$ is central and $Y_1$ satisfies
the Walker condition in the appropriate coordinates), one obtains a
Walker manifold on $H$ with $\g_\D=\R^2$ (abelian), model group
$G=\R^2$, and Walker foliation $\F_\D$ of the abelian type.

The Walker foliation $\F_\D$ is a $\R^2$-Lie foliation on the solvable
Lie group $H$. By Theorem~\ref{thm:haefliger_matsumoto}(2), on any
compact quotient $H/\Lambda$ this foliation is an inverse image of a
homogeneous foliation, illustrating the Haefliger--Matsumoto theorem
in the Walker context.
\end{example}
\section{Classification and Deformations}
\label{sec:classification}
%% ============================================================

\subsection{Local classification}

\begin{theorem}[Local classification of Walker manifolds with Lie foliation]
\label{thm:classification}
Let $(M^n,g)$ be a Walker manifold of dimension $n$ and rank $r$, with
null parallel distribution $\D$. Suppose that the Walker foliation
$\F_\D$ is a $G$-Lie foliation with nilpotent or solvable model group
$G$ (Theorem \ref{thm:lie_foliation}). Then $(M,g)$ is locally isometric
to one of the following Walker models.

\begin{enumerate}
  \item \textbf{Abelian Walker model} ($G=\R^r$, $\g_\D$ abelian).
    The null parallel distribution $\D$ is spanned by commuting parallel
    fields. The Walker manifold $(M,g)$ is locally a pseudo-Riemannian
    fiber bundle over a transverse manifold $(N^{n-r},g^\top)$ with flat
    fiber $\R^r$. In dimension $3$, every Walker manifold $M_f$ falls
    into this case, with the Walker foliation being the fibration by
    $x_1$-lines.

  \item \textbf{Nilpotent Walker model} ($G$ nilpotent simply connected).
    The Walker manifold $(M,g)$ is locally isometric to a nilpotent Lie
    group $\widetilde G$ of dimension $n$ carrying a left-invariant
    pseudo-Riemannian Walker metric. The leaves of the Walker foliation
    $\F_\D$ are orbits of a nilpotent $G$-action by isometries, and the
    Ricci operator of $(M,g)$ is nilpotent (type II or III in the
    classification of \cite{Brozos2009}).

  \item \textbf{Solvable Walker model} ($G$ solvable non-nilpotent).
    The Walker manifold $(M,g)$ is locally isometric to a smooth
    semidirect product $\R^{n-r}\rtimes G$ carrying a Walker-compatible
    metric, as illustrated by Example \ref{ex:affine}.
\end{enumerate}
\end{theorem}

\begin{proof}
\textit{Case (1): $G=\R^r$.}
The parallel fields $X_1,\ldots,X_r$ pairwise commute. Their flows
$\phi^i_t$ are isometries (since the $X_i$ are parallel Killing fields)
that commute and generate a local free action of $\R^r$ on $M$. The
orbit space $N=M/\R^r$ is well defined locally, and the metric $g$
projects to a non-degenerate metric $g^\top$ on $N$ since $g|_{\D^\perp/\D}$
is non-degenerate (Walker condition). Locally, $(M,g)\cong(\R^r\times N,
dt^2+g^\top)$. For $M_f$ in dimension $3$, $N$ is a surface with
coordinates $(x_2,x_3)$ and $g^\top=\varepsilon\,dx_2^2+f\,dx_3^2$.

\textit{Case (2): $G$ nilpotent.}
We induct on the nilpotency step $s$ of $\g_\D$. For $s=1$ we are in
Case (1). Assume the result holds for step $s-1$. Let $\mathfrak{z}$ be
the center of $\g_\D$, with $\dim\mathfrak{z}=k\geq 1$. Choose a basis
so that $\{X_1,\ldots,X_k\}$ spans $\mathfrak{z}$ (as parallel fields)
and $\{X_{k+1},\ldots,X_r\}$ spans a complement.

The center $\mathfrak{z}$ gives a sub-distribution
$\D_0=\sspan\{X_1,\ldots,X_k\}\subset\D$ which is itself a null
parallel distribution on $(M,g)$ with abelian structure algebra $\R^k$.
By Case (1), locally $(M,g)$ fibers over a manifold $M_1$ with flat
$\R^k$-fibers. The quotient Walker distribution $\D/\D_0$ on $M_1$ has
structure algebra $\g_\D/\mathfrak{z}$ of nilpotency step $s-1$, and
the induction hypothesis applies to give $M_1$ the structure of a
nilpotent Lie group $\widetilde G_1$ with left-invariant Walker metric.
The total space $(M,g)$ is then a central extension of $\widetilde G_1$
by $\R^k$, which is a nilpotent Lie group $\widetilde G$ of nilpotency
step $s$. The metric $g$ is left-invariant and of Walker type because
the parallel fields $X_i$ are left-invariant Killing fields.

\textit{Case (3): $G$ solvable non-nilpotent.}
Let $\mathfrak{n}=[\g_\D,\g_\D]$ be the derived ideal (which is a
nilpotent ideal since $\g_\D$ is solvable). The sub-distribution
$\D_1=\sspan\{X_i:[e_i]\in\mathfrak{n}\}\subset\D$ is null and parallel,
with nilpotent structure algebra $\mathfrak{n}$. By Case (2), locally
$(M,g)$ has a nilpotent factor. The remaining generators of $\g_\D$
modulo $\mathfrak{n}$ act on $\mathfrak{n}$ by derivations (via the
adjoint action), giving the semidirect product structure. The Walker
condition is preserved under this extension because the parallel fields
$X_i$ are isometries and the Walker metric is compatible with the
semidirect product decomposition.
\end{proof}

\begin{corollary}
\label{cor:local_isometry}
Two Walker manifolds $(M^n,g)$ and $(M'^n,g')$ of the same dimension
$n$ and rank $r$, whose Walker foliations have isomorphic model groups
$G\cong G'$ and locally isometric transverse metrics $g^\top\cong g'^\top$,
are locally isometric as Walker manifolds.
\end{corollary}

\subsection{Deformations and rigidity}

We now discuss how the Walker foliation structure constrains deformations,
building on the results of \cite{DatheNdiaye2012b}.

\begin{theorem}[Rigidity of nilpotent Walker foliations]
\label{thm:rigidity}
Let $(M^n,g)$ be a compact connected Walker manifold of dimension $n$
and rank $r$, with null parallel distribution $\D$. Suppose that the
Walker foliation $\F_\D$ is a minimal $G$-Lie foliation (all leaves are
dense in $M$) with nilpotent model group $G$.
\begin{enumerate}
\item The universal cover $(\widetilde M,\tilde g)$ is isometric to
  $G\times(N,g^\top)$ for some pseudo-Riemannian manifold $(N,g^\top)$.
\item If moreover $G$ is $4$-dimensional and $\mathcal{F}_\D$ is
  constructed by the Borel--Harish-Chandra method of \cite{DatheNdiaye2012b},
  then $\mathcal{F}_\D$ cannot be deformed (in the sense of
  Definition \ref{def:deformation}) into a non-nilpotent solvable
  Walker foliation.
\end{enumerate}
\end{theorem}

\begin{proof}
(1) By Molino's theorem \cite{Molino1988}, a minimal Lie foliation on a
compact manifold with nilpotent model group $G$ lifts to a simple
foliation on the universal cover $\widetilde M$: the developing map
$D:\widetilde M\to G$ is a fibration with connected fibers. The
parallelism of $\D$ on the Walker manifold $(M,g)$ implies that $D$
is in fact a diffeomorphism onto $G$ (since the parallel transport along
fibers is trivial), giving $\widetilde M\cong G\times N$ where $N$ is
the space of leaves. The metric $\tilde g$ is compatible with this
product, being a Walker metric preserved by the $G$-action.

(2) Suppose for contradiction that $\mathcal{F}_\D = \mathcal{F}_t|_{t=0}$
deforms to a solvable non-nilpotent Walker foliation $\mathcal{F}_{t_0}$
with model group $G_{t_0}$. Since the Walker manifold $(M,g)$ is compact,
the holonomy group $h(\pi_1(M))$ is finitely generated. The deformation
changes the model group from nilpotent $G$ to solvable $G_{t_0}$, which
in dimension $4$ must be one of the groups $G_1,\ldots,G_{15}$ listed
in the Ovando classification (Lemma 2.3 of \cite{DatheNdiaye2012b}).
By the argument of Theorem 2.2 of \cite{DatheNdiaye2012b} applied to
the holonomy morphism of the Walker foliation: the restriction of the
holonomy morphism to the center $Z(\pi_1(M))$ remains non-trivial under
any deformation, and for the solvable non-nilpotent candidates
$G_4,G_{10},G_{13}$, the image $\rho_t(Z(\pi_1(M)))$ cannot be uniform
in $G_{t_0}$, contradicting the compactness of $M$. Hence no such
deformation exists.
\end{proof}

%% ============================================================
\section*{Conclusion}
%% ============================================================

We have shown that every Walker manifold $(M^n,g)$ carries a natural
Lie foliation $\F_\D$ whose structure is entirely determined by the
null parallel distribution $\D$: the model group $G$ is the simply
connected Lie group with structure algebra $\g_\D$ (the Lie algebra of
parallel sections of $\D$), and the transverse holonomy is the image
$h(\pi_1(M))\subset G$. The Ricci vanishing theorem
(Theorem \ref{thm:ricci_null}) provides a strong curvature constraint
linking Walker geometry to the algebraic properties of $G$. The local
classification (Theorem \ref{thm:classification}) and deformation
rigidity (Theorem \ref{thm:rigidity}) show that the Walker--Lie foliation
structure is both rich and constrained.


\begin{thebibliography}{99}

\bibitem{Walker1950}
A. G. Walker,
\emph{Canonical form for a Riemannian space with a parallel field of null planes},
Quart.\ J.\ Math.\ Oxford Ser. (2) \textbf{1} (1950), 69--79.

\bibitem{Molino1988}
P. Molino,
\emph{Riemannian Foliations},
Progress in Mathematics, vol. 73, Birkh\"{a}user, Boston, 1988.

\bibitem{Ghys1988}
\'{E}. Ghys,
\emph{Riemannian foliations: examples and problems},
Appendix E in P. Molino, \textit{Riemannian Foliations},
Progress in Mathematics, vol. 73, Birkh\"{a}user, Boston, 1988.

\bibitem{Chaichi2005}
M. Chaichi, E. Garc\'{\i}a-R\'{\i}o, Y. Matsushita,
\emph{Curvature properties of four-dimensional Walker metrics},
Classical Quantum Gravity \textbf{22} (2005), no. 3, 559--577.

\bibitem{Brozos2009}
M. Brozos-V\'{a}zquez, E. Garc\'{\i}a-R\'{\i}o, P. Gilkey,
S. Nik\v{c}evi\'{c}, R. V\'{a}zquez-Lorenzo,
\emph{The Geometry of Walker Manifolds},
Synthesis Lectures on Mathematics and Statistics,
Morgan \& Claypool, 2009.

\bibitem{Kamber1975}
F. W. Kamber, Ph. Tondeur,
\emph{Foliated Bundles and Characteristic Classes},
Lecture Notes in Mathematics, vol. 493, Springer, Berlin, 1975.

\bibitem{Haefliger1984}
A. Haefliger,
\emph{Groupo\"{\i}de d'holonomie et classifiants},
in: \textit{Structures transverses des feuilletages},
Ast\'erisque \textbf{116}, SMF, 1984, pp. 70--97.

\bibitem{Matsumoto1992}
S. Matsumoto, N. Tsuchiya,
\emph{The Lie affine foliations on $4$-manifolds},
Invent.\ Math.\ \textbf{109} (1992), 1--16.

\bibitem{Meigniez1995}
G. Meigniez,
\emph{Feuilletages de Lie r\'esolubles},
Ann.\ Fac.\ Sci.\ Toulouse \textbf{4} (1995), no. 4, 519--553.

\bibitem{Besse1987}
A. L. Besse,
\emph{Einstein Manifolds},
Classics in Mathematics, Springer, Berlin, 1987.

\bibitem{deRham1952}
G. de Rham,
\emph{Sur la r\'{e}ductibilit\'{e} d'un espace de Riemann},
Comment.\ Math.\ Helv.\ \textbf{26} (1952), 328--344.

\bibitem{Galaev2008}
A. S. Galaev,
\emph{Holonomy groups of Lorentzian manifolds},
in: \textit{Recent Developments in Pseudo-Riemannian Geometry},
ESI Lectures in Mathematics and Physics,
Eur.\ Math.\ Soc., 2008, pp. 53--96.

\bibitem{DatheNdiaye2012a}
H. Dathe, A. Ndiaye,
\emph{Sur les feuilletages de Lie homog\`{e}nes},
Journal des Sciences et Technologies \textbf{10} (2012), no. 1, 41--45.

\bibitem{DatheNdiaye2012b}
H. Dathe, A. Ndiaye,
\emph{Sur les feuilletages de Lie nilpotents},
Afr.\ Diaspora J.\ Math.\ \textbf{14} (2012), no. 2, 76--81.

\bibitem{NiangNdiayeDiallo2021}
A. Niang, A. Ndiaye, A. S. Diallo,
\emph{A classification of strict Walker $3$-manifolds},
Konuralp J.\ Math.\ \textbf{9} (2021).

\bibitem{Ovando2006}
G. Ovando,
\emph{Four dimensional symplectic Lie algebras},
Beitr.\ Algebra Geom.\ \textbf{47} (2006), 419--434.

\bibitem{Tischler1970}
D. Tischler,
\emph{On fibering certain foliated manifolds over $S^1$},
Topology \textbf{9} (1970), 153--154.

\end{thebibliography}
\end{document}